\RequirePackage{fix-cm}
\documentclass[smallextended]{svjour3}       
\smartqed  
\usepackage{graphicx,amssymb,amsmath,url}
\usepackage{hyperref}
\usepackage{color}
\usepackage{mathptmx}
\usepackage{latexsym}
\spnewtheorem{thm}{Theorem}[section]{\bfseries}{\upshape}
\numberwithin{thm}{section}
\spnewtheorem{cor}[thm]{Corollary}{\bfseries}{\upshape}
\spnewtheorem{lem}[thm]{Lemma}{\bfseries}{\upshape}
\spnewtheorem{prop}[thm]{Proposition}{\bfseries}{\upshape}
\spnewtheorem{defn}[thm]{Definition}{\bfseries}{\upshape}
\spnewtheorem{rem}[thm]{Remark}{\bfseries}{\upshape}
\spnewtheorem{exam}[thm]{Example}{\bfseries}{\upshape}
\numberwithin{equation}{section}
\newcommand{\seq}[1]{\langle #1\rangle}
\DeclareMathOperator{\Cay}{Cay}
\DeclareMathOperator{\Aut}{Aut}
\def\G{\Gamma}
\begin{document}

\title{Distance-regular Cayley graphs with least eigenvalue $-2$}

\dedication{In honor of Andries Brouwer for his 65th birthday}



\author{Alireza Abdollahi \and Edwin R. van Dam   \and  Mojtaba Jazaeri}


\institute{A. Abdollahi \at
              Department of Mathematics, University of Isfahan, Isfahan 81746-73441, Iran; School of Mathematics, Institute for Research in Fundamental Sciences (IPM), P.O. Box: 19395-5746, Tehran, Iran \\
              \email{a.abdollahi@math.ui.ac.ir}           
           \and
           E. R. van Dam \at
              Department of Econometrics and O.R., Tilburg University, P.O. Box 90153, 5000 LE Tilburg, The Netherlands \\
              \email{Edwin.vanDam@uvt.nl}
              \and
              M. Jazaeri \at
              Department of Mathematics, Faculty of Mathematics and Computer Sciences, Shahid Chamran University of Ahvaz, Ahvaz, Iran; School of Mathematics, Institute for Research in Fundamental Sciences (IPM), P.O. Box: 19395-5746, Tehran, Iran \\
              \email{M.Jazaeri@scu.ac.ir, SeJa81@gmail.com}
}

\date{Received: date / Accepted: date}

\maketitle

\begin{abstract}
We classify the distance-regular Cayley graphs with least eigenvalue $-2$ and diameter at most three. Besides sporadic examples, these comprise of the lattice graphs, certain triangular graphs, and line graphs of incidence graphs of certain projective planes. In addition, we classify the possible connection sets for the lattice graphs and obtain some results on the structure of distance-regular Cayley line graphs of incidence graphs of generalized polygons.
\end{abstract}

\section{Introduction}
Distance-regular graphs form an important class of graphs in the area of algebraic graph theory. Originally, they were defined as a generalization of distance-transitive graphs, and many of them are not even vertex-transitive. For background on distance-regular graphs, we refer to the monograph by Brouwer, Cohen, and Neumaier \cite{BCN} and the recent survey by Van Dam, Koolen, and Tanaka \cite{DKT}. Here we study the question which distance-regular graphs are Cayley graphs. This question has been well-studied for distance-regular graphs with diameter two, that is, for strongly regular graphs, see the survey paper on partial difference sets by Ma \cite{Ma}. Miklavi\v{c} and Poto\v{c}nik \cite{MP1,MP2} classified the distance-regular circulant graphs and distance-regular Cayley graphs on dihedral groups, whereas Miklavi\v{c} and \v{S}parl \cite{MS} studied a particular class of distance-regular Cayley graphs on abelian groups. See also the monograph by Konstantinova \cite{EK} for some basic facts and problems on Cayley graphs and distance-regular graphs.

It is well-known that graphs with least eigenvalue $-2$ have been classified by using root lattices, see \cite[\S 3.12]{BCN}. In particular, it follows that a distance-regular graph with least eigenvalue $-2$ is strongly regular or the line graph of a regular graph with girth at least five. The strongly regular graphs with least eigenvalue $-2$ have been classified by Seidel \cite{Se}. We will give an overview of which of these graphs is a Cayley graph and in particular, we will classify the possible connection sets for the lattice graphs, using some general results that we obtain for the distance-regular line graphs of incidence graphs of generalized polygons. We will also classify the Cayley graphs with diameter three among the distance-regular line graphs, in particular the line graphs of Moore graphs and the line graphs of incidence graphs of projective planes. What remains open is to classify which line graphs of incidence graphs of generalized quadrangles and hexagons are Cayley graphs.
\section{Preliminaries}
Let $G$ be a  finite group with identity element $e$ and $S\subseteq G\setminus\{e\}$ be a set such that $S=S^{-1}$ (we call $S$ inverse-closed). An (undirected) Cayley graph $\Cay(G,S)$ with connection set $S$ is the graph whose vertex set is $G$ and where two vertices $a$ and $b$ are adjacent (denoted by $a \sim b$) whenever $ab^{-1}\in S$. The Cayley graph $\Cay(G,S)$ is connected if and only if the subgroup $\langle S \rangle$ generated by $S$ is equal to $G$. In the literature, it is sometimes assumed explicitly that a Cayley graph is connected. In this case, the connection set is also called a generating set. Here we follow the terminology used by Alspach \cite{BW}. We denote the order of an element $a \in G$ by $O(a)$, the subgroup generated by $a$ by $\seq{a}$ and the cyclic group of order $n$ by $\mathbb{Z}_{n}$. Furthermore, the cycle graph of order $m$ is denoted by $C_{m}$ and the line graph of a graph $\Gamma$ by $L(\Gamma)$.
\subsection{Distance-regular graphs}
A strongly regular graph with parameters $(v,k,\lambda,\mu)$ is a $k$-regular graph with $v$ vertices such that every pair of adjacent vertices has $\lambda$ common neighbors and every pair of non-adjacent vertices has $\mu$ common neighbors. Here we exclude disjoint unions of complete graphs and edgeless graphs, and therefore strongly regular graphs are connected with diameter two.

A connected graph with diameter $d$ is distance-regular whenever for all vertices $x$ and $y$, and all integers $i,j\leq d$, the number of vertices at distance $i$ from $x$ and distance $j$ from $y$ depends only on $i$, $j$, and the distance between $x$ and $y$.
A distance-regular graph with diameter two is the same as a strongly regular graph.

A generalized $d$-gon is a point-line incidence structure whose (bipartite) incidence graph has diameter $d$ and girth $2d$. It is of order $(s,t)$ if every line contains $s+1$ points, and every point is on $t+1$ lines. For $s=t$, both the incidence graph and its line graph are distance-regular. This line graph can also be viewed as the point graph of a generalized $2d$-gon of order $(s,1)$. For some basic background on generalized polygons, we refer to the monographs by Godsil and Royle \cite[\S 5.6]{Go} and Brouwer, Cohen, and Neumaier \cite[\S 6.5]{BCN}.

The (adjacency) spectrum of a graph is the multiset of eigenvalues of its adjacency matrix. As mentioned in the introduction, distance-regular graphs with least eigenvalue $-2$ can be classified. In particular, we have the following.

\begin{thm}\label{thm:drg-2} \cite[Thm.~3.12.4 and 4.2.16]{BCN} Let $\G$ be a distance-regular graph with least eigenvalue $-2$. Then $\G$ is a cycle of even length, or its diameter $d$ equals $2,3,4,$ or $6$. Moreover,
\begin{itemize}
\item If $d=2$, then $\G$ is a cocktail party graph, a triangular graph, a lattice graph, the Petersen graph, the Clebsch graph, the Shrikhande graph, the Schl\"{a}fli graph, or one of the three Chang graphs,
   \item If $d=3$, then $\G$ is the line graph of the Petersen graph, the line graph of the Hoffman-Singleton graph, the line graph of a strongly regular graph with parameters $(3250,57,0,1)$, or the line graph of the incidence graph of a projective plane,
   \item If $d=4$, then $\G$ is the line graph of the incidence graph of a generalized quadrangle of order $(q,q)$,
   \item If $d=6$, then $\G$ is the line graph of the incidence graph of a generalized hexagon of order $(q,q)$.
   \end{itemize}
\end{thm}

Recall that the triangular graph $T(n)$ is the line graph of the complete graph $K_{n}$, the lattice graph $L_{2}(n)$ is the line graph of the complete bipartite graph $K_{n,n}$ (a generalized $2$-gon), and the cocktail party graph $CP(n)$ is the complete multipartite graph with $n$ parts of size two. Note also that a projective plane is a generalized $3$-gon.

We note that the distance-regular graphs with least eigenvalue {\em larger} than $-2$ are also known. Besides the complete graphs (with least eigenvalue $-1$), there are the cycles of odd length, and these are clearly Cayley graphs.

 \subsection{Vertex-transitivity and edge-transitivity}

Recall that a graph $\Gamma$ is vertex-transitive whenever the automorphism group of $\Gamma$ acts transitivity on the vertex set of $\Gamma$, i.e. if $x$ is a fixed vertex of $\Gamma$, then $\{x^{\sigma}|\sigma \in \Aut(\Gamma)\}$ is equal to the set of vertices of $\Gamma$. It is clear that Cayley graphs are vertex-transitive. In fact, a graph $\Gamma$ is a Cayley graph if and only if the automorphism group $\Aut(\Gamma)$ of $\Gamma$ contains a regular subgroup, see \cite[Thm.~2.2]{BW}.

A graph $\Gamma$ is called edge-transitive whenever the automorphism group of $\Gamma$ acts transitivity on the edge set of the graph. Because line graphs play an important role in this paper, also the concept of edge-transitivity is relevant. Indeed, the following result provides us with a connection between the vertex-transitivity of the line graph of a graph $\Gamma$ and the edge-transitivity of $\Gamma$.

 \begin{thm} \cite[Thm.~5.3]{Sa} \label{line isomorphism}
 Let $\Gamma$ be a connected graph which is not isomorphic to the complete graphs $K_{2}$, $K_{4}$, a triangle with an extra edge attached, and two triangles sharing an edge. Then the automorphism group of $\Gamma$ and its line graph are isomorphic, with the natural group isomorphism $\varphi : \Aut(\Gamma) \rightarrow \Aut(L(\Gamma))$, defined by $\varphi(\sigma)=\widetilde{\sigma}$  for $\sigma \in \Aut(\Gamma)$,
where $\widetilde{\sigma}$ acts on the line graph of $\Gamma$ such that $\widetilde{\sigma}(\{v,w\})=\{\sigma(v),\sigma(w)\}$, where $v$ and $w$ are adjacent in $\Gamma$.
 \end{thm}
 \begin{lem}\label{lem:edgevertextransitive}
A connected regular graph is edge-transitive if and only if its line graph is vertex-transitive.
 \end{lem}
 \begin{proof}
 Let $\Gamma$ be connected and regular. If $\Gamma$ is isomorphic to $K_{2}$ or $K_{4}$, then $\Gamma$ is edge-transitive and the line graph of $\Gamma$ is vertex-transitive. On the other hand, if $\Gamma$ is not isomorphic to $K_{2}$ or $K_{4}$, then the automorphism group of $\Gamma$ and its line graph are isomorphic with the natural group isomorphism by Theorem \ref{line isomorphism}, which completes the proof. \qed
 \end{proof}

\subsection{Groups and products}

Two subgroups $H$ and $K$ in $G$ are conjugate whenever there exists an element $g \in G$ such that $K=g^{-1}Hg$. The semidirect product $G$ of a group $N$ by a group $H$ is denoted by $H\ltimes N$ or $N\rtimes H$. It has the property that it contains a normal subgroup $N_1$ isomorphic to $N$ and a subgroup $H_1$ isomorphic to $H$ such that $G=N_1H_1$ and $N_1\cap H_1=\{e\}$.

Let $G$ be a finite group with subgroups $H$ and $K$ such that $G=HK$ and the intersection of $H$ and $K$ is the identity of $G$. Then $G$ is called a general product of $H$ and $K$ (see \cite{Co}).

Finally we mention a result that we will use in Section \ref{sec:projplane}.

\begin{thm} \cite[Thm.~9.1.2]{R} \label{Hall}
 Let $N$ be a normal subgroup of a finite group $G$, and let $n=|N|$, and $m=[G:N]$. Suppose that $n$ and $m$ are relatively prime. Then $G$ contains subgroups of order $m$ and any two of them are conjugate in $G$.
\end{thm}

\section{Some results on generalized polygons}\label{sec:genpol}

Let $\G$ be a distance-regular line graph of the incidence graph of a generalized $d$-gon of order $(q,q)$. Then $\G$ can also be seen as the point graph of a generalized $2d$-gon of order $(q,1)$. It follows that each vertex of $\G$ is contained in two maximal cliques, of size $q+1$, and every edge of $\G$ is contained in a unique maximal clique. Thus, $\G$ does not have $K_{1,3}$ nor $K_{1,2,1}$ as an induced subgraph. Moreover, $\G$ has diameter $d$ and every induced cycle in $\G$ is either a $3$-cycle or a $2d$-cycle. We will use these properties to derive some general results on the structure of the connection set in case $\G$ is a Cayley graph.

\begin{thm} \label{Cayley line structure}
Let $d \geq 2$, let $\Gamma$ be the line graph of the incidence graph of a generalized $d$-gon of order $(q,q)$, and suppose that $\Gamma$ is a Cayley graph $\Cay(G,S)$. Then there exist two subgroups $H$ and $K$ of $G$ such that $S=(H \cup K) \setminus \{e\}$, with $|H|=|K|=q+1$ and $H \cap K=\{e\}$ if and only if $\seq{a} \subseteq S \cup \{e\}$ for every element $a$ of order $2d$ in $S$.
\end{thm}

\begin{proof} One direction is clear: if there are subgroups $H$ and $K$ of $G$ such that $S=(H \cup K) \setminus \{e\}$, then $\seq{a} \subseteq S \cup \{e\}$ for every element $a$ in $S$. To show the other direction, assume that $\seq{a} \subseteq S \cup \{e\}$ for every element $a$ of order $2d$ in $S$.

We first claim that $\seq{a} \subseteq S \cup \{e\}$ for {\em all} $a \in S$. In order to prove this, let $a\in S$ and $n=O(a) \neq 2d$. If $n=2$ or $3$, then $\seq{a} \subseteq S \cup \{e\}$ since $S=S^{-1}$. If $n \geq 4$, then it is clear that the induced subgraph $\G_{\seq{a}}$ of $\G$ on $\seq{a}$ contains a cycle $e \sim a \sim a^2 \sim \cdots \sim a^{n-1} \sim e$ of length $n$ (see also \cite[Lemma 2.6]{AV}). Because $n \neq 3$ and $n\neq 2d$, it follows that this cycle is not an induced cycle. Thus, there must be an extra edge in $\G_{\seq{a}}$, that is, an edge that is not generated by $a$ or $a^{-1}$, and hence $a^i \in S$ for some $i$ with $1<i<n-1$. Now $e$ is adjacent to $a,a^{-1}$, and $a^i$, and because $\G$ does not contain an induced subgraph $K_{1,3}$, it follows that $a^2 \in S$, or $a^{i-1}\in S$, or $a^{i+1}\in S$. Let us consider the case that $a^{i-1} \in S$, with $i>2$. By considering the induced subgraph on $\{e,a,a^{i-1},a^{i}\}$, it follows that $a^{i-2}\in S$ because $\G$ does not contain an induced subgraph $K_{1,2,1}$. Similarly, by considering the induced subgraph on $\{e,a^{-1},a^{i-1},a^{i}\}$, it follows that $a^{i+1}\in S$. By repeating this argument, it follows that $\seq{a} \subseteq S \cup \{e\}$. The other cases go similarly, which proves our claim.

Let $H$ and $K$ be the two cliques of size $q+1$ that contain $e$. Then $S=(H \cup K) \setminus \{e\}$ and $H \cap K=\{e\}$. What remains to be shown is that $H$ and $K$ are subgroups of $G$.

Let $a\in H \setminus \{e\}$. Because the graph induced on $\seq{a}$ is a clique, and there are no edges between $H \setminus \{e\}$ and $K \setminus \{e\}$, it follows that $\seq{a} \subseteq H$, In particular, $a^{-1}\in H$.

Now let $a,b \in H$, and let us show that $ba^{-1} \in H$, thus showing that $H$ is a subgroup of $G$. If $a=b$, $a=e$, or $b=e$, then this clearly implies that $ba^{-1} \in H$. In the other cases, we have that $b \sim a$, so $ba^{-1} \in S$. Because $ba^{-1} \sim a^{-1}$, and there are no edges between $H \setminus \{e\}$ and $K \setminus \{e\}$, it follows that $ba^{-1} \in H$. Thus, $H$ --- and similarly $K$ --- is a subgroup of $G$. \qed
\end{proof}

The condition that $\seq{a} \subseteq S \cup \{e\}$ for every element $a$ of order $2d$ in $S$ is not redundant. Indeed, the lattice graph $L_2(2)$, which is the line graph of $K_{2,2}$ (the incidence graph of a generalized $2$-gon of order $(1,1)$) is isomorphic to the Cayley graph $\Cay(\mathbb{Z}_4,\{\pm1\})$. Both elements $a$ in $S =\{\pm1\}$ have order $4$, but $a^2 \notin S$, and indeed $S \cup \{e\}$ does not contain a nontrivial subgroup of $\mathbb{Z}_4$.

The proof of Theorem \ref{Cayley line structure} indicates that the condition $\seq{a} \subseteq S \cup \{e\}$ for every element $a$ of order $2d$ in $S$ can be replaced by the condition that $a^2 \in S$ for every element $a$ of order $2d$ in $S$. We can in fact generalize this as follows.

\begin{cor}
Let $d \geq 2$, let $\Gamma$ be the line graph of the incidence graph of a generalized $d$-gon of order $(q,q)$, and suppose that $\Gamma$ is a Cayley graph $\Cay(G,S)$. Then there exist two subgroups $H$ and $K$ of $G$ such that $S=(H \cup K) \setminus \{e\}$, with $|H|=|K|=q+1$ and $H \cap K=\{e\}$ if and only if for every element $a$ of order $2d$, there exists an element $s \in S$ such that $s \neq a$, $a^{-1}$ and $sas^{-1} \in S$.
\end{cor}

\begin{proof}
Let $a$ be of order $2d$ in $S$, and assume that there exists an element $s \in S$ such that $s \neq a$, $a^{-1}$ and $sas^{-1} \in S$. By Theorem \ref{Cayley line structure}, it suffices to prove that $\seq{a} \subseteq S \cup \{e\}$. Because $e$ is adjacent to $a$, $a^{-1}$, and $s$, and $G$ has no induced subgraph $K_{1,3}$, it follows that there is at least one edge within $\{a,a^{-1},s\}$. If $a$ and $a^{-1}$ are adjacent, then $a^2 \in S$. Because $\G$ does not contain an induced subgraph $K_{1,2,1}$, it then follows by induction and by considering the induced subgraph on $\{e,a,a^{i},a^{i+1}\}$ (for $i \geq 2$) that $\seq{a} \subseteq S \cup \{e\}$. So let us assume that $a$ and $a^{-1}$ are not adjacent. Without loss of generality, we may thus assume that $s$ is adjacent $a^{-1}$, and hence that $sa \in S$. Now $e$ is adjacent to $sa$, $a$ and $s$. Furthermore, $sa$ is adjacent to $a$ and $s$ since $sas^{-1} \in S$. It follows, again because $\G$ does not contain an induced subgraph $K_{1,2,1}$, that $a$ is adjacent to $s$. Using the same argument once more gives that $a$ and $a^{-1}$ are adjacent, which is a contradiction that finishes the proof. \qed
\end{proof}

\begin{rem}
In view of the above, if there exists an element $a \in S$ of order $2d$ such that $\seq{a} \nsubseteq S \cup \{e\}$, then $a$ and $a^{-1}$ are not adjacent. We may therefore assume that $a \in H$ and $a^{-1} \in K$, where $H$ and $K$ are the two maximal cliques (but not subgroups) containing $e$. But clearly the set $Ka$ is a maximal clique containing $a$ and $e$. Because every edge is in a unique maximal clique, it follows that $Ka = H$. Therefore, in this case, $S=(K \cup Ka)\setminus \{e\}$. In the case of the Cayley graph $\Cay(\mathbb{Z}_4,\{\pm1\})$, we indeed have $K=\{-1,0\}$ and $H=K+1$.
\end{rem}

As a first application of the above, we obtain that the (distance-regular) line graph of the Tutte-Coxeter graph is not a Cayley graph.

 \begin{prop} \label{tuttecoxeter}
 The line graph of the Tutte-Coxeter graph is not a Cayley graph.
 \end{prop}
\begin{proof}The Tutte-Coxeter graph is the incidence graph of a generalized quadrangle ($4$-gon) of order $(2,2)$. It has $30$ vertices and $45$ edges. If its line graph is a Cayley graph $Cay(G,S)$, then $|G|=45$ and $|S|=4$. Because $G$ has no element of order $8$, it follows from Theorem \ref{Cayley line structure} that there exist two subgroups $H$ and $K$ of $G$ such that $S=(H \cup K) \setminus \{e\}$, where $|H|=|K|=3$ and $H \cap K=\{e\}$. Furthermore, the group $G$ is an abelian
 group isomorphic to $\mathbb{Z}_{3} \times \mathbb{Z}_{3} \times \mathbb{Z}_{5}$ or $\mathbb{Z}_{9} \times \mathbb{Z}_{5}$ since $G$ has only one subgroup of order $9$ and one subgroup of order $5$ by Sylow's theorems. By the structure of the connection set $S$, it now follows that $G$ must be the abelian group isomorphic to $\mathbb{Z}_{3} \times \mathbb{Z}_{3} \times \mathbb{Z}_{5}$ but in this case the Cayley graph $Cay(G,S)$ is not connected, a contradiction. Therefore the line graph of the Tutte-Coxeter graph is not a Cayley graph. \qed
\end{proof}

We finish this section with a result that shows that the obtained structure of $S$ in the above fits naturally with line graphs of bipartite graphs.

\begin{lem}\label{lem:line}
Let $\G$ be a Cayley graph $\Cay(G,S)$, where $S=(H \cup K) \setminus \{e\}$ for nontrivial subgroups $H$ and $K$ of $G$ such that $H \cap K=\{e\}$. Then $\G$ is the line graph of a bipartite graph.
\end{lem}

\begin{proof}
From the structure of $S$, it follows that each vertex is in two maximal cliques, and every edge is in a unique maximal clique. By a result of Krausz \cite{Kr} (see \cite[Thm.~7.1.16]{W}) it follows that $\G$ is a line graph of a graph, $\G'$, say. The graph $\G'$ has the maximal cliques of $\G$ as vertices, and two such cliques are adjacent in $\G'$ if and only if they intersect; the corresponding edge in $\G'$ is the vertex in $\G$ that is contained in both cliques.

Because $S=(H \cup K) \setminus \{e\}$ and $H \cap K=\{e\}$, we can distinguish between two kinds of maximal cliques. We call such a clique an $H$-clique if the edges in the clique are generated by an element in $H$, and the other cliques are similarly called $K$-cliques. Now it is clear that every edge in $\G'$ has one vertex in the set of $H$-cliques and the other vertex in the set of $K$-cliques. Thus $\G'$ is bipartite. \qed
\end{proof}
\section{Strongly regular graphs}\label{sec:srg}
In this section, we will determine which strongly regular graphs with least eigenvalue $-2$ are Cayley graphs, using the case of diameter $d=2$ in the classification given in Theorem \ref{thm:drg-2}.

\subsection{The sporadic graphs}

Besides the three infinite families of strongly regular graphs with least eigenvalue $-2$, we have to consider the Petersen graph, the Clebsch graph, the Shrikhande graph, the Schl\"{a}fli graph, and the Chang graphs. The Petersen graph is the unique strongly regular graph with parameters $(10,3,0,1)$. It is the complement of the line graph of the complete graph $K_{5}$, and therefore it is not a Cayley graph, by Corollary \ref{triangular} below (see also \cite[Lemma 3.1.3]{Go}).

\begin{prop}(Folklore). \label{Petersen_no}
The Petersen graph is not a Cayley graph.
\end{prop}

It is well-known that the complement of the Clebsch graph is the folded $5$-cube, which is strongly regular with parameters $(16,5,0,2)$ (see \cite[p.~119]{BH}). The $d$-dimensional cube $Q_{d}$ is the distance-regular graph whose vertex set can be labeled with the $2^{d}$ binary $d$-tuples such that two vertices are adjacent whenever their labels differ in exactly one position (clearly this is a Cayley graph). The folded $d$-cube is the distance-regular graph that can be obtained from the cube $Q_{d-1}$ by adding a perfect matching that connects vertices at distance $d-1$ (see \cite{van Bon}). It is evident that the folded $d$-cube is the Cayley graph $\Cay(G,S)$, where $G$ is the elementary abelian $2$-group of order $2^{d-1}$ and
\begin{equation*}
S=\{(1,0,0,\ldots,0), (0,1,0,\ldots,0), \ldots , (0,0,0,\ldots,0,1), (1,1,\ldots,1)\}.
\end{equation*}
Thus, the Clebsch graph is a Cayley graph.

The Shrikhande graph is a strongly regular graph with the same parameters as the lattice graph $L_{2}(4)$ and can be constructed as a Cayley graph
\begin{equation*}
 \Cay(\mathbb{Z}_{4} \times \mathbb{Z}_{4}, \{\pm (0,1), \pm (1,0), \pm (1,-1)\}).
\end{equation*}
This construction `on the torus' is accredited to Biggs \cite{Biggs} by Gol'fand, Ivanov, and Klin \cite[p.~182]{GIK}.

The Schl\"{a}fli graph is the unique strongly regular graphs with parameters $(27,16,10,8)$. It follows from the work by Liebeck, Praeger, and Saxl \cite{LPS} (see also \cite[Lemma 2.6]{LPM}) that it is a Cayley graph over the semidirect product $\mathbb{Z}_{9} \rtimes \mathbb{Z}_{3}$. Using {\sf GAP} \cite{GAP}, we checked that with $G=\mathbb{Z}_{9} \rtimes \mathbb{Z}_{3}=\seq{a,b|a^9=b^3=1,b^{-1}ab=a^7}$ and $S=\{a,a^8,a^3,a^6,b,b^2,a^7b,a^5b^2,a^2b,a^4b^2\}$, the Cayley graph $\G=\Cay(G,S)$ indeed is the complement of the Schl\"{a}fli graph. Note that $\G$ is also the point graph of the unique generalized quadrangle of order $(2,4)$, with lines thus being the triangles in $\G$. Therefore, these lines can be obtained as the right `cosets' of the five triangles $\{e,a,a^2b\}$, $\{e,a^8,a^7b\}$,$\{e,a^3,a^6\}$,$\{e,b,b^2\}$,$\{e,a^5b^2,a^4b^2\}$ through $e$.

The Schl\"{a}fli graph can also be constructed as a Cayley graph over $(\mathbb{Z}_{3} \times \mathbb{Z}_{3}) \rtimes \mathbb{Z}_{3}$, the other nonabelian group of order 27. Indeed, we again checked with {\sf GAP} \cite{GAP} that $\G=\Cay(G',S')$ for $G'= \seq{a,b,c|a^3=b^3=c^3=e, abc=ba,ac=ca,bc=cb}$ and $S'=\{a,a^2,b,b^2,c,c^2,cba,a^2b^2c^2,aba,bab\}$. In this case all nonidentity elements of the group have order $3$, and hence the triangles through $e$ are subgroups $H_1,\dots,H_5$ of $G'$, with trivial intersection and $S'=(H_1\cup\cdots\cup H_5) \setminus \{e\}$ (cf.~Theorem \ref{Cayley line structure}). Again, the cosets of these subgroups give the lines of the generalized quadrangle of order $(2,4)$. From the above, we conclude the following.

\begin{prop} \label{sporadicsrg} The Clebsch graph, the Shrikhande graph, and the Schl\"{a}fli graph are Cayley graphs.
\end{prop}

The Chang graphs are strongly regular graphs with the same parameters as the line graph of the complete graph $K_{8}$. These three graphs can be obtained by Seidel switching in $L(K_{8})$. According to \cite{Math2}, the orders of the automorphism groups of these graphs are $384$, $360$, and $96$, respectively.
\begin{prop} \label{chang}
The three Chang graphs are not Cayley graphs.
\end{prop}
\begin{proof}
Let $\Gamma$ be one of the Chang graphs, and suppose on the contrary that it is a Cayley graph, and hence that it is vertex-transitive. Let $x$ be a fixed vertex in $\Gamma$. Then the order of $\{x^{\sigma}| \sigma \in \Aut(\Gamma)\}$ is equal to $28$ since $\Gamma$ is vertex-transitive. It follows that the index of $\Aut(\Gamma)$ over the stabilizer of $x$ is $28$. Therefore $28$ must divide the order of $\Aut(\Gamma)$, which is a contradiction. \qed
\end{proof}

\subsection{The infinite families}

A cocktail party graph is a complete multipartite graph with parts of size two, and clearly such a graph is a Cayley graph. By \cite[Prop.~2.6]{AJ1}, we obtain the following result.
\begin{prop} \label{cocktail}
A Cayley graph $\Cay(G,S)$ is a cocktail party graph if and only if $G$ has an element $a$ of order $2$ and $S=G \setminus \seq{a}$ .
\end{prop}
Consider a set $X$ of size $n$ and let $V$ be the collection of all subsets of size $m$ in $X$, with $m \geq 2$ and $n \geq 2m+1$. The Kneser graph $K(n,m)$ is the graph with vertex set $V$ such that two vertices $A$ and $B$ in $V$ are adjacent whenever $|A \cap B|=0$. The Kneser graph $K(n,2)$ is the complement of the triangular graph $T(n)$. Godsil \cite{G} characterized the Cayley graphs among the Kneser graphs.
\begin{thm} \cite{G}
Except in the following cases, the Kneser graph $K(n,m)$ is not a Cayley graph.
\begin{itemize}
\item $m=2$, $n$ is a prime power and $n\equiv 3 \text{\em{ (mod 4)}}$,
\item $m=3$, $n=8$ or $n=32$.
\end{itemize}
\end{thm}

As a corollary, we obtain a result first obtained by Sabidussi \cite{SaVTG}. Note that the triangular graphs $T(2)$ and $T(3)$ are complete graphs, and that $T(4)$ is isomorphic to the cocktail party graph $CP(3)$.
\begin{cor} \cite{SaVTG} \label{triangular}
The triangular graph $T(n)$ is a Cayley graph if and only if $n=2,3,4$ or $n\equiv 3 \text{\em{ (mod 4)}}$ and $n$ is a prime power.
\end{cor}

Godsil \cite{mathoverflow} gave the following construction of the triangular graph $T(n)$ as a Cayley graph $\Cay(G,S)$ for prime powers $n\equiv 3$~(mod 4).
Let $\mathbb{F}$ be the field of order $n$. For $a,b \in \mathbb{F}$, let the map $T_{a,b}: \mathbb{F} \rightarrow \mathbb{F}$ be defined by $T_{a,b}(x)=ax+b$.
Let $G$ be the group of maps $T_{a,b}$, with $a$ a non-zero square and $b$ arbitrary. It is not hard to see that $G$ acts regularity on the edges of the complete graph $K_{n}$ (with vertex set $\mathbb{F}$), using that $-1$ is a non-square (whence the assumption that $n\equiv 3$~(mod 4)). As connection set $S$ one can take the set of maps $T_{a,b} \in G$ such that either $T_{a,b}(0) \in \{0,1\}$ or $T_{a,b}(1) \in \{0,1\}$ (thus mapping the vertex $\{0,1\}$ of the triangular graph to an adjacent vertex).

As a final family of strongly regular graphs with least eigenvalue $-2$, we consider the lattice graphs. Let $n \geq 2$. The lattice graph $L_2(n)$ is the line graph of the complete bipartite graph $K_{n,n}$. It is isomorphic to the Cartesian product of two complete graphs $K_n$, and hence to the Cayley graph
 $\Cay(\mathbb{Z}_{n} \times \mathbb{Z}_{n},\{(0,1),\ldots,(0,n-1),(1,0),\ldots,(n-1,0)\})$. Because $K_{n,n}$ is the incidence graph of a generalized $2$-gon, we can apply the results of Section \ref{sec:genpol}. We will use these to give a characterization of the lattice graphs as Cayley graphs.
\begin{thm} \label{lattice}
  Let $n \geq 2$, let $G$ be a finite group, $S$ be an inverse-closed subset of $G$, and let $\G=\Cay(G,S)$. Then the following hold:
  \begin{itemize}
  \item
  If $G$ is a general product of two of its subgroups $H$ and $K$ of order $n$ and $S=(H \cup K)\setminus \{e\}$, then $\G$ is isomorphic to the lattice graph $L_{2}(n)$,
  \item
  If $\G$ is isomorphic to the lattice graph $L_{2}(n)$ and $\seq{a} \subseteq S \cup \{e\}$ for every element $a$ of order $4$ in $S$, then $G$ is a general product of two of its subgroups $H$ and $K$ of order $n$ and $S=(H \cup K)\setminus \{e\}$.
  \end{itemize}
\end{thm}
\begin{proof}
Let $G$ be a general product of two of its subgroups $H$ and $K$ of order $n$ and let $S=(H \cup K)\setminus \{e\}$. Then $|G|=n^2$. By using the results in Section \ref{sec:genpol}, we know that every vertex in $\G$ is in two maximal cliques of size $n$. A simple counting argument shows that there are $2|G|/n=2n$ maximal cliques in $\G$. By Lemma \ref{lem:line} and its proof, it follows that $\G$ is the line graph of a bipartite graph $\G'$ on the $2n$ maximal cliques of $\G$, and that each clique is on $n$ edges in $\G'$. This implies that $\G'$ is the complete bipartite graph $K_{n,n}$, and hence $\G$ is the lattice graph $L_2(n)$.

To prove the second item, suppose that $\G$ is isomorphic to the lattice graph $L_{2}(n)$ and $\seq{a} \subseteq S \cup \{e\}$ for every element $a$ of order $4$ in $S$. It follows by Theorem \ref{Cayley line structure} that there are subgroups $H$ and $K$ of order $n$ in $G$ such that $H \cap K=\{e\}$ and $S=(H \cup K)\setminus \{e\}$. Now $K$ is a maximal clique in $\G$. Let $g$ be a vertex not in $K$. Then the structure of the lattice graph implies that $g$ is adjacent to precisely one vertex $k \in K$. Thus $gk^{-1} \in S$, and hence it follows that $gk^{-1} \in H$ (because if it were in $K$, then so would $g$), so $g=hk$ for some $h \in H$. Therefore $G$ is the general product of $H$ and $K$, which completes the proof. \qed
\end{proof}
We recall from Section \ref{sec:genpol} that the lattice graph $L_2(2)$ is isomorphic to the Cayley graph $\Cay(\mathbb{Z}_4,\{\pm1\})$, which is an example such that $G=\mathbb{Z}_4$ cannot be written as a general product $HK$ with inverse-closed sets $H$ and $K$ of size $2$.

We now conclude this section by giving the classification of all strongly regular Cayley graphs with least eigenvalue at least $-2$ (which follows from the above). Recall that the only strongly regular graph with least eigenvalue larger than $-2$ is the $5$-cycle.
\begin{thm}
A graph $\Gamma$ is a strongly regular Cayley graph with least eigenvalue at least $-2$ if and only if $\Gamma$ is isomorphic to one of the following graphs.
\begin{itemize}
\item The cycle $C_{5}$, the Clebsch graph, the Shrikhande graph, or the Schl\"{a}fli graph,
\item The cocktail party graph $CP(n)$, with $n\geq 2$,
\item The triangular graph $T(n)$, with $n=4$, or $n\equiv 3 \text{\em{ (mod 4)}}$ and $n$ a prime power, $n>4$,
\item The lattice graph $L_{2}(n)$, with $n\geq 2$.
\end{itemize}
\end{thm}
\section{Distance regular graphs with diameter three}\label{sec:diameter3}
In this section, we will determine which distance-regular graphs with least eigenvalue $-2$ and diameter three are Cayley graphs. By the classification given in Theorem \ref{thm:drg-2}, we again have to consider a few sporadic examples and an infinite family.

\subsection{The line graphs of Moore graphs}

\begin{prop} \label{Petersen}
The line graph of the Petersen graph is not a Cayley graph.
\end{prop}
\begin{proof}
Let $\G$ be the line graph of the Petersen graph, and suppose that $\G \cong \Cay(G,S)$, hence $|G|=15$ and $|S|=4$. Therefore there exists a subgroup of order $15$ of the automorphism group of $\G$ which acts transitively on the edges of the Petersen graph. By Sylow's theorems, it is easy to see that the only group of order $15$ is the cyclic group $\mathbb{Z}_{15}$. This abelian group $G$ acts transitively on the edges of the Petersen graph, and because this graph is not bipartite, it follows that $G$ acts transitively on the vertices of the Petersen graph (cf.~\cite[Lemma 3.2.1]{Go}). But every transitive abelian group acts regularly (cf.~\cite[Prop.~16.5]{Biggs}), which gives a contradiction because the Petersen graph does not have 15 vertices.    \qed
\end{proof}
\begin{prop} \label{Hoffman}
The line graph of the Hoffman-Singleton graph is not a Cayley graph.
\end{prop}
\begin{proof}
Let $\G$ be the line graph of the Hoffman-Singleton graph, and suppose that $\G \cong \Cay(G,S)$, hence $|G|=175$ and $|S|=12$. It is easy to see that there exist only two groups of order $175$ by Sylow's theorems, which are the abelian groups $\mathbb{Z}_{175}$ and $\mathbb{Z}_{35} \times \mathbb{Z}_{5}$. The result now follows similarly as in Proposition \ref{Petersen} \qed
\end{proof}

The final case in this section is the line graph of a putative Moore graph on $3250$ vertices.

\begin{prop} \label{unknown}
The line graph of a strongly regular graph with parameters $(3250,57,0,1)$ is not a Cayley graph.
\end{prop}
\begin{proof}
Let $\Gamma$ be a strongly regular graph with parameters $(3250,57,0,1)$ and suppose that the line graph of $\Gamma$ is a Cayley graph. Then $L(\Gamma)$ is vertex-transitive and therefore $\Gamma$ is edge-transitive by Lemma \ref{lem:edgevertextransitive}. On the other hand, it is known that $\Gamma$ is not vertex-transitive, see \cite[Prop.~11.2]{BH}, and therefore $\Gamma$ must be bipartite by \cite[Lemma 3.2.1]{Go}, which is a contradiction. \qed
\end{proof}

\subsection{The line graphs of the incidence graphs of projective planes}\label{sec:projplane}

Recall that a projective plane of order $q$ is a point-line incidence structure such that each line has $q+1$ points, each point is on $q+1$ lines, and every pair of points in on a unique line. It is the same as a generalized $3$-gon of order $(q,q)$ and a $2$-$(q^2+q+1,q+1,1)$ design. Currently, projective planes of order $q$ are only known to exist for prime powers $q$, and for $q=1$. For $q>1$, the classical construction of a projective plane of order $q$ uses the finite field GF($q$) and gives the so-called Desarguesian plane of order $q$. We note that Loz, Ma\v{c}aj, Miller, \v{S}iagiov\'{a}, \v{S}ir\'{a}\v{n}, and Tomanov\'{a} \cite{LMM} showed that the (distance-regular) incidence graph of a Desarguesian plane is a Cayley graph. Here we will consider the line graph, however. For $q=1$, the line graph of the incidence graph is a $6$-cycle, which is a Cayley graph. We therefore assume from now on that $q>1$. We note that the dual incidence structure of a projective plane is also a projective plane; if a projective plane is isomorphic to its dual, then we say it is self-dual.

Consider now a projective plane $\pi$ of order $q$, and let $\G_{\pi}$ be the incidence graph of $\pi$. Recall from Theorem \ref{line isomorphism} that the automorphism group of $\G_{\pi}$ and its line graph $L(\G_{\pi})$ are isomorphic. A collineation (automorphism) of $\pi$ is a permutation of the points and lines that maps points to points, lines to lines, and that preserves incidence. If $\pi$ is not self-dual, then an automorphism of the incidence graph $\G_{\pi}$ must be a collineation.  Additionally, if the projective plane is self-dual, then the automorphism group of $\G_{\pi}$ has index $2$ over the automorphism group of $\pi$; in this case the plane has so-called correlations (isomorphisms between the plane and its dual; see also \cite{M}) on top of collineations.

By construction, a vertex in $L(\G_{\pi})$ corresponds to an incident point-line pair --- also called flag -- of $\pi$. If $L(\G_{\pi})$ is a Cayley graph (or more generally, is vertex-transitive), then we have a group of collineations and correlations of $\pi$ that is transitive on flags. In particular, we have the following lemma.

\begin{lem} Let $\pi$ be a projective plane of order $q$, with $q$ even. If $L(\G_{\pi})$ is a Cayley graph, then $\pi$ has a collineation group acting regularly on its flags.
\end{lem}

\begin{proof} If $L(\G_{\pi})$ is a Cayley graph, then there must be a group $G$ of automorphisms of $\G_{\pi}$ acting regularly on the edges of $\G_{\pi}$. The group $G$ therefore has order $(q+1)(q^2+q+1)$. Moreover, $G$ is (isomorphic to) a group of collineations and correlations of $\pi$ that acts regularly on its flags. If this group contains correlations, then it has an index 2 subgroup of collineations, but this is impossible because the order of $G$ is odd. Hence $\pi$ has a collineation group acting regularly on its flags. \qed
\end{proof}

For $q$ even, we can therefore use the following characterization by Kantor \cite{K}.

\begin{thm} \cite[Thm.~A]{K} \label{Kantor}
Let $q \geq 2$, let $\pi$ be a projective plane of order $q$, and let $F$ be a collineation group of $\pi$ that is transitive on flags. Then either
\begin{itemize}
\item $PSL(3,q)$ is contained in $F$ and $\pi$ is Desarguesian, or
\item $F$ is a  Frobenius group of odd order $(q+1)(q^2+q+1)$, and $q^2+q+1$ is prime.
\end{itemize}
\end{thm}
Recall that $PSL(3,q)$ is the projective special linear group, which has order
\begin{equation*}
\frac{q^3(q^{3}-1)(q^{2}-1)}{\gcd(3,q-1)}.
\end{equation*}
If $L(\G_{\pi})$ is a Cayley graph $\Cay(G,S)$, then $|G|=(q^2+q+1)(q+1)$, and the action of $G$ on the flags of $\pi$ must be regular. Because the order of $PSL(3,q)$ is larger than $|G|$, it follows that $G$ is a Frobenius group of odd order $(q^2+q+1)(q+1)$, and that $q^2+q+1$ is prime. Recall that
a Frobenius group is a group $F$ which has a non-trivial subgroup $H$ such that $H \cap x^{-1}Hx = \{e\}$ for all $x \in F \setminus H$. Furthermore, $N= F \setminus \bigcup_{x \in F}(x^{-1}Hx\setminus \{e\})$ is a normal subgroup of $F$ such that $F=HN$ and  $H \cap K = \{e\}$, i.e. $F$ is the semidirect product $N \rtimes H$ (see \cite{R}).
\begin{prop} \label{incidence}
If the line graph of the incidence graph of a projective plane $\pi$ of order $q$ is a Cayley graph $\Cay(G,S)$, where $G$ corresponds to a group of collineations of $\pi$, then $G$ is $N \rtimes H$ in which $N$ is a normal subgroup of prime order $q^2+q+1$ and $H$ is a subgroup of odd order $q+1$.
\end{prop}
\begin{proof}
It follows from the above that $G$ is a  Frobenius group of odd order $(q^2+q+1)(q+1)$, and $q^2+q+1$ is a prime number. It follows that $G$ has a normal $(q^2+q+1)$-Sylow subgroup $N$ of order $q^2+q+1$ by Sylow's theorems. On the other hand, there exists a subgroup $H$ of order $q+1$ in $G$ by Theorem \ref{Hall}, and the intersection of $N$ and $H$ is the identity element of $G$. Therefore $G$ is $N \rtimes H$. \qed
\end{proof}

It is widely believed that there is no non-Desarguesian plane admitting a collineation group acting transitively on flags. Thas and Zagier \cite{TZ} showed that if such a plane exists, then its order is at least $2 \times 10^{11}$.

On the other hand, Higman and McLaughlin \cite{HM} showed that the only Desarguesian planes admitting a collineation group acting regularly on flags are those of order $2$ and $8$. Indeed, the line graphs of the incidence graphs of these projective planes can be constructed as Cayley graphs as follows:
 \begin{exam}
The Heawood graph is the incidence graph of the Fano plane; its line graph is the unique graph with spectrum $\{4^1,(1+\sqrt{2})^6,(1-\sqrt{2})^6,-2^8\}$ (see \cite{VH}).
Let $G=\mathbb{Z}_{7} \rtimes \mathbb{Z}_{3}=\seq{a,b|a^{7}=b^{3}=e,b^{-1}ab=a^2}$. Let $H=\seq{b}$, $K=\seq{a^{-1}ba}$ and $S=(H \cup K) \setminus \{e\}$ (cf.~Theorem \ref{Cayley line structure}). By using {\sf GAP} \cite{GAP} and similar codes as in \cite[p. 4]{AJ2}, it is checked that the Cayley graph $\Cay(G,S)$ is indeed the line graph of the Heawood graph.

Similarly the line graph of the incidence graph of the (unique) projective plane of order $8$ is obtained by taking $G=\mathbb{Z}_{73} \rtimes \mathbb{Z}_{9}=\seq{a,b|a^{73}=b^{9}=e,b^{-1}ab=a^2}$, $H=\seq{b}$, $K=\seq{a^{-1}ba}$, and $S=(H \cup K) \setminus \{e\}$.
\end{exam}

We may thus conclude the following.

 \begin{thm}\label{cayley3}
 Let $\Gamma$ be a distance-regular Cayley graph with diameter three and least eigenvalue at least $-2$. Then $\Gamma$ is isomorphic to one of the following graphs.
 \begin{itemize}
 \item The cycle $C_6$ or $C_7$,
 \item The line graph of the incidence graph of the Desarguesian projective plane of order $2$ or $8$,
 \item The line graph of the incidence graph of a non-Desarguesian projective plane of order $q$, where $q^2+q+1$ is prime and $q$ is even and at least $2 \times 10^{11}$,
 \item The line graph of the incidence graph of a projective plane of odd order with a group of collineations and correlations acting regularly on its flags.
 \end{itemize}
 \end{thm}

It would be interesting to find out whether any of the results on collineations of projective planes can be extended to groups of collineations and correlations, and thus rule out the final case of Theorem \ref{cayley3}. We could not find any such results in the literature.

Besides the line graph of the Tutte-Coxeter graph (see Proposition \ref{tuttecoxeter}) we leave the case of the line graphs of incidence graphs of generalized quadrangles and hexagons open (cf.~Theorem \ref{thm:drg-2}). For some results on flag-transitive generalized quadrangles, we refer to Bamberg, Giudici, Morris, Royle, and Spiga \cite{flagquad}; for flag-transitive generalized hexagons, we refer to Schneider and Van Maldeghem \cite{flaghexa}.
\begin{acknowledgements}
The authors thank Brendan McKay for pointing to \cite[Thm.~5.3]{Sa} in order to prove Lemma \ref{lem:edgevertextransitive}.
 Mojtaba Jazaeri thanks the Graduate Studies of University of Isfahan and Tilburg University since this paper was partly written during his visit at Tilburg University as part of his PhD program in Isfahan. The research of Alireza Abdollahi was in part supported by a grant from School of Mathematics, Institute for Research in Fundamental Sciences (IPM) (No. 94050219). Alireza Abdollahi is also supported financially by the Center of Excellence, University of Isfahan. The research of Mojtaba Jazaeri was in part supported by a grant from School of Mathematics, Institute for Research in Fundamental Sciences (IPM) (No. 94050039).
\end{acknowledgements}



\end{document}